\documentclass[11pt]{amsart}
\usepackage[cm]{fullpage}
\usepackage{pgfplots}

\usepackage{amssymb,amsmath,amsthm,amscd,mathrsfs,graphicx}
\usepackage[cmtip,all]{xy}
\usepackage{pb-diagram}
\usepackage{tikz}

\numberwithin{equation}{section}

\newtheorem{theorem}{Theorem}[section]
\newtheorem{proposition}{Proposition}[section]
\newtheorem{lemma}{Lemma}[section]
\newtheorem{corollary}{Corollary}[section]

\newcommand{\ov}[1]{\overline{#1}}

\newcommand{\ve}{\varepsilon}

\theoremstyle{definition}

\theoremstyle{remark}

\begin{document}
\bibliographystyle{amsplain}

\title{The insulated conductivity problem, \\ effective gradient estimates \\ and the maximum principle}

\author[B. Weinkove]{Ben Weinkove}
\address{Department of Mathematics, Northwestern University, 2033 Sheridan Road, Evanston, IL 60208, USA.}

\thanks{Research supported in part by   NSF grant DMS-2005311. Part of this work was carried out while the author was visiting the Department of Mathematical Sciences at the University of Memphis whom he thanks for  their kind support and hospitality.}

\begin{abstract}
We consider  the insulated conductivity problem with two unit balls as insulating inclusions, a distance of order $\varepsilon$ apart.  The solution $u$ represents the electric potential.  In dimensions $n \ge 3$ it is an open problem to find the optimal bound on the gradient of $u$, the electric field,  in the narrow region between the insulating bodies.    Li-Yang recently proved a bound of order $\varepsilon^{-(1-\gamma)/2}$ for some $\gamma>0$.   In this paper we use a direct maximum principle argument to sharpen the Li-Yang estimate for $n \ge 4$.   Our method gives effective lower bounds on $\gamma$, which in particular approach $1$ as $n$ tends to infinity.
\end{abstract}

\maketitle

\section{Introduction}

We consider the following problem.  Let $B^+, B^- \subset \mathbb{R}^n$, for $n\ge 2$,  be two closed balls of radius 1 centered at $(0, \ldots, 0, 1+\ve)$ and $(0, \ldots, 0, -1-\ve)$ respectively.  Let $\Omega \subset \mathbb{R}^n$ be a bounded domain containing the convex hull of $B^+$ and $B^-$.  Assume that the boundary $\partial \Omega$  is $C^{1,\alpha}$, for a fixed $0<\alpha<1$,  Define $\Omega'= \Omega \setminus (B^+ \cup B^-)$.   See Figure \ref{figure1}.

Let $u: \overline{\Omega'} \rightarrow \mathbb{R}$ be the solution  of the \emph{insulated conductivity problem}:
\begin{equation} \label{mainequation}
\left\{ \begin{split} \Delta u = {} & 0, \quad \textrm{in } \Omega' \\ \partial_{\nu} u = {} & 0, \quad \textrm{on } \partial B^+ \cup \partial B^- \\
u= {} & \varphi, \quad \textrm{on } \partial \Omega, \end{split} \right.
\end{equation}
for a given function $\varphi \in C^{1,\alpha}(\partial \Omega)$.   Here $\partial_{\nu}$ denotes the normal derivative on $\partial B^+\cup \partial B^-$.  It is well-known that there is a unique such solution $u$ which is smooth on $\Omega'\cup \partial B^+\cup \partial B^-$  and $C^{1, \alpha}$ on $\overline{\Omega'}$.

The solution $u$ represents the electric potential in the domain $\Omega'$ with insulating inclusions $B^+$ and $B^-$.  There is interest from the perspective of engineering  in estimating the magnitude of the electric field $\nabla u$ in the narrow region between  $B^+$ and $B^-$ .   
  An open problem is to find the optimal upper bound on $|\nabla u|_{L^{\infty}(\Omega')}$ in terms of $\ve$.   We refer the reader to \cite{BASL, BV, BC, K, LN, LV, M}
 and the references therein for background and early work on this problem.
  
Ammari-Kang-Lee-Lee-Lim \cite{AKLLL} proved in dimension $n=2$ an upper bound for $|\nabla u|$ of the order $\ve^{-1/2}$, although with the condition $u=\varphi$ on $\partial \Omega$ replaced by a condition of being asymptotic to a harmonic function at infinity.  Later, in the context of (\ref{mainequation}), Bao-Li-Yin \cite{BLY} proved the following estimate for $n \ge2$:
\begin{equation} \label{BLY}
| \nabla u |(x)\le  \frac{C\| \varphi \|_{C^{1,\alpha}(\partial \Omega)} }{(\ve+|x'|^2)^{1/2}}, \quad \textrm{for } x\in \Omega',
\end{equation}
where $C$ depends only on $n$ and $\Omega$.   Here we are writing
$$x' = (x_1, \ldots, x_{n-1}), \quad |x'| = \sqrt{x_1^2+ \cdots + x_{n-1}^2}.$$

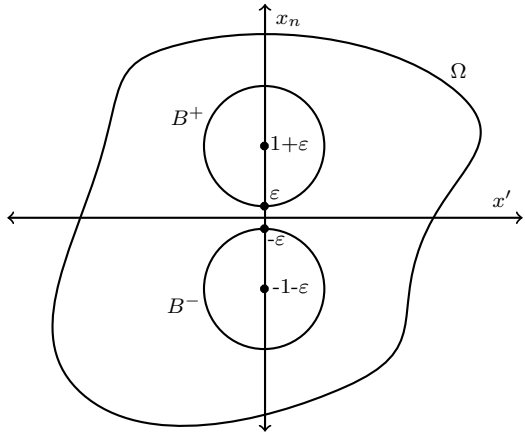
\begin{figure}[h] 
\begin{tikzpicture}
    \begin{axis}[thick,
        xmin=-4.5,xmax=4.5,
        ymin=-3.5,ymax=3.5,
       axis x line=middle,
       axis y line=middle,
        axis line style=<->,
        xlabel={$\scriptstyle x'$},
        x tick label style={major tick length=0pt},
        ylabel={$\scriptstyle x_n$},
        yticklabels={,,},
        x tick label style={major tick length=0pt},      
       xticklabels={,,}]
       \draw (3.415cm,3.8cm) circle (.8cm);
 \draw (3.415cm, 1.9cm) circle (.8cm);
 \draw [->] plot [smooth cycle, tension=0.9] coordinates { (1cm,.5cm) (4.5cm, .6cm) (5.5cm,2.5cm) (6cm, 4.5cm) (2.5cm,5.2cm) (1.2cm,3.5cm)};
      \end{axis}
    \draw [ultra thick] (3.417, 3.0) circle (.03cm);
    \node(A) at (3.56, 3.15cm) {$\scriptstyle\varepsilon$};
        \draw [ultra thick] (3.417, 2.7) circle (.03cm);
            \node(B) at (3.75, 3.82cm) {$\scriptstyle1+\varepsilon$};
        \draw [ultra thick] (3.417, 3.8) circle (.03cm);
    \node(C) at (3.58, 2.54cm) {$\scriptstyle\textrm{-}\varepsilon$};
     \draw [ultra thick] (3.417, 1.9) circle (.03cm);
    \node(C) at (3.77, 1.93cm) {$\scriptstyle\textrm{-}1\textrm{-}\varepsilon$};
\node(D) at (2.4, 4.2cm) {$\scriptstyle B^+$};
\node(E) at (2.35, 1.7cm) {$\scriptstyle B^-$};
\node(F) at (6cm, 4.8cm) {$\scriptstyle \Omega$};
\end{tikzpicture}
\caption{The domain $\Omega$ containing the balls $B^+$ and $B^-$.} \label{figure1}
\end{figure}

 Ammari-Kang-Lim \cite{AKL} showed that for $n=2$, an upper bound of order $\ve^{-1/2}$ cannot be improved, and hence (\ref{BLY}) is in that sense optimal for $n=2$.  See also \cite{ACKLY, AKLLL, DL, KLY, LYu} for work on sharper estimates, asymptotics and discussion in the case $n=2$.

However (\ref{BLY}) is not optimal for $n \ge 3$. Indeed, Yun \cite{Y3} proved an upper bound for $n=3$ of the form $|\nabla u| \le C\ve^{\frac{\sqrt{2} - 2}{2}}$ on the shortest segment connecting the two balls, under a condition that $u$ is asymptotic to a harmonic function at infinity.  He also showed that this estimate is optimal for a special choice of harmonic function.  For (\ref{mainequation}), Li-Yang \cite{LY} showed that for dimensions $n \ge 3$ the estimate (\ref{BLY}) can be sharpened to 
\begin{equation} \label{LY}
| \nabla u |(x)\le  \frac{C\| \varphi \|_{C^{1,\alpha}(\partial \Omega)} }{(\ve+|x'|^2)^{(1-\gamma)/2}}, \quad \textrm{for } x\in \Omega',
\end{equation}
for some (non-explicit) $\gamma>0$.  The optimal $\gamma$ remains unknown.

We remark that intuitively one should indeed expect stronger bounds for larger dimensions $n\ge 3$ compared to $n=2$.  In dimension two, after taking the limit as $\ve \rightarrow 0$ so that $B^+$ and $B^-$ touch, points  in $\Omega'$ with $x_1<0$ and $x_1>0$ may be close in $\mathbb{R}^2$ but ``far away''  in $\Omega'$. This is a phenomenon that does not occur for $n \ge 3$, where the extra dimensions allow for $u$ to diffuse around the intersection point.

We prove an effective version of the Li-Yang result in dimension at least four.  For $n \ge 4$, let $\gamma^*=\gamma^*(n)$ be the positive solution of the quadratic equation:
\begin{equation} \label{qe}
(n-2) (\gamma^*)^2 + (n^2-4n+5) \gamma^* - (n^2-5n+5)=0.
\end{equation}
One can check that $0<\gamma^*<1$ and $\gamma^*(n)$ increases to $1$ as $n\rightarrow \infty$.  On the other hand, (\ref{qe}) has no positive solutions for $n =2$ or $3$.  

Our main result shows that (\ref{LY}) holds for $0<\gamma<\gamma^*$.

\begin{theorem} \label{mainthm} Let $n \ge 4$ and let $u$ solve (\ref{mainequation}) as above.  Let $0<\gamma<\gamma^*(n)$.  Then 
  there exists a uniform constant $C$ depending only on $n$, $\Omega$ and $\gamma$ such that 
\begin{equation} \label{LY2}
| \nabla u |(x)\le  \frac{C\| \varphi \|_{C^{1,\alpha}(\partial \Omega)} }{(\ve+|x'|^2)^{(1-\gamma)/2}}, \quad \textrm{for } x\in \Omega'.
\end{equation}

\end{theorem}

In the following table we give the exact and approximate numerical values $\gamma^*$ for $n=4,5,6,7$:

\medskip
\renewcommand{\arraystretch}{1.2}
\centerline{\begin{tabular}{c | c | c}
$n$ &  $\gamma^*$  & approx. \\
\hline
4 &  $\frac{1}{4}(\sqrt{33}-5)$ &  0.186 \\
5 &  $\frac{1}{3}(2\sqrt{10}-5)$ & 0.442\\
6 & $\frac{1}{8}(\sqrt{465}-17)$  & 0.570 \\
7 & $\frac{1}{5}(2 \sqrt{66}-13)$ & 0.650
 \end{tabular}}

\medskip

Our method is a direct application of the maximum principle, and can be used to give another proof of (\ref{BLY}), which we explain in Section \ref{sectionBLY} below.  Our approach differs from the works  \cite{BLY, LY} whose techniques involve dividing the domain into a large number of subdomains.  

We have no particular reason to believe that our bounds for $\gamma$ are optimal.  Indeed since our results do not recover the estimate (\ref{LY}) of Li-Yang in dimension three it is possible that our methods can be further refined.

There are various ways in which the above setup can also be generalized, as in \cite{BLY0, BLY,  LY,  Y1, Y2}:  $B^+$ and $B^-$ can be replaced by more general domains, say with strongly convex boundaries near the origin, and the equation $\Delta u=0$ can be replaced by a more general elliptic equation or system.  Although we expect our method applies also to these settings we restrict ourselves to the setup above for the sake of simplicity and to more cleanly illustrate our technique.

We now briefly describe the idea of our method.  We apply the maximum principle to a specific quantity in the narrow region between the insulators.  The quantity suggested by the estimate (\ref{LY}) is $(\ve+|x'|^2)^{(1-\gamma)}|\nabla u|^2$.  However to rule out the maximum occurring at points on $\partial B^+$ or $\partial B^-$, we  subtract a term of the form $x_n^{2-\gamma}$ whose normal derivative along this boundary has a good sign.  Unfortunately, its second derivative blows up to the order $x_n^{-\gamma}$ at an interior maximum, giving a bad term for the maximum principle near $x_n=0$.  We make an appropriate adjustment, considering the quantity:
$$( |x'|^{2-2\gamma} + \ve^{1-\gamma(1-\delta)} - A(bx_n^2+|x'|^4)^{1-\gamma/2} )|\nabla u|^2,$$
for a small $\delta>0$ and constants $A, b>0$ which are chosen to optimize the estimate.  In fact, since this quantity is not smooth at points with $x'=0$ we introduce a positive constant $\sigma>0$ to regularize it, and then later let $\sigma$ tend to zero.    We call this modified quantity $Q$ (see (\ref{Q}) below).  

In Section \ref{sectionproof} below we show that if $Q$ achieves a local maximum at a boundary point of $B^+$, say, close to the origin, then at this point,
$$0 \le \partial_{\nu} Q \le F|\nabla u|^2,$$
where $\nu$ is the inward unit normal to the ball and $F$ is a certain function involving $x, A, b$ and $\gamma$.    At an interior maximum of $Q$, close to the origin, we show that, for a different function $G$,
$$0 \ge \Delta Q \ge G |\nabla u|^2.$$
It turns out that with careful choices of $A, b$ and $\gamma$ we can show that $F<0$ and $G>0$, ruling out both of the cases above, but only when 
$$\rho:= -(n-2) \gamma^2 - (n^2-4n+5) \gamma + (n^2-5n+5)>0.$$
This restricts our result to $n \ge 4$.

 The outline of the paper is as follows.  In Section \ref{sectionprelim} we recall some well-known and elementary results.  In Section \ref{sectionproof} we give the proof of Theorem \ref{mainthm}.  Finally, in Section \ref{sectionBLY} we explain briefly how our methods can give a direct maximum principle proof of (\ref{BLY}).  

\medskip
\noindent
{\em Note.} \ Seven months after this paper was completed and posted on the arXiv, Dong-Li-Yang \cite{DLY} posted a preprint establishing  optimal estimates (in terms of the constant $\gamma$) for the insulated conductivity problem in all dimensions $n \ge 3$.  The techniques of \cite{DLY} are completely different from those introduced here.  It would be interesting to know whether the maximum principle techniques of the current paper  can be used to give another proof of the optimal Dong-Li-Yang estimates.

\section{Preliminaries} \label{sectionprelim}

Let $u$ solve (\ref{mainequation}) as in the introduction.  For a small constant $c>0$, let $B_c$ be the closed ball of radius $c$ centered at the origin.  Define $V = \ov{\Omega'\cap B_c}$ (see Figure \ref{fig2}).  This is a narrow region between the inclusions $B^+$ and $B^-$ and is the set on which  we will later need to prove estimates for $|\nabla u|$.

\begin{figure}[h]
\begin{tikzpicture}
    \begin{axis}[thick,
        xmin=-4.5,xmax=4.5,
        ymin=-3.5,ymax=3.5,
       axis x line=middle,
       axis y line=middle,
        axis line style=<->,
        xlabel={$\scriptstyle x'$},
        x tick label style={major tick length=0pt},
        ylabel={$\scriptstyle x_n$},
        yticklabels={,,},
        x tick label style={major tick length=0pt},      
       xticklabels={,,}]
       \draw (3.415cm,4cm) circle (1cm);
 \draw (3.415cm, 1.7cm) circle (1cm);
      \end{axis}
    \node(A) at (3.56, 3.15cm) {$\scriptstyle\varepsilon$};
    \node(C) at (3.58, 2.54cm) {$\scriptstyle\textrm{-}\varepsilon$};
\node(D) at (2.1, 4.2cm) {$\scriptstyle B^+$};
\node(E) at (2.1, 1.7cm) {$\scriptstyle B^-$};
\node(F) at (4.43, 2.7) {$\scriptstyle c$};
\node(G) at (2.25, 2.7) {$\scriptstyle - c$};
\node(H) at (4.37, 3.25) {$\scriptstyle V$};
\def\firstcircle{(3.415cm,4cm) circle (1cm)}
\def\secondcircle{(3.415, 2.85) circle (0.9cm)}
\def\thirdcircle{(3.415cm, 1.7cm) circle (1cm)}
\begin{scope}
\clip \thirdcircle (0,0cm) rectangle (5cm, 5cm);
\clip \firstcircle (0,0cm) rectangle (5cm, 5cm);
\draw[fill=gray]  \secondcircle;
\end{scope}
\end{tikzpicture}
\caption{The set $V= \overline{\Omega' \cap B_c}$, shown as the shaded region.}  \label{fig2}
\end{figure}
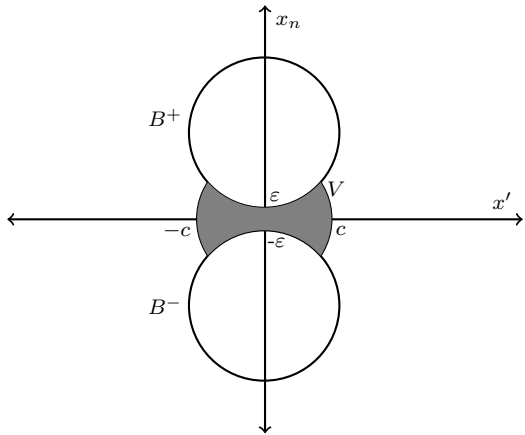

The $L^{\infty}$ norm of $u$ is bounded on $\Omega'$, and  the gradient of $u$ is bounded away from the region $V$.   More precisely, we have the following result, whose proof is well-known.

\begin{proposition}  \label{pwk} We have
\begin{enumerate}
\item[(i)] $\displaystyle{\| u \|_{L^{\infty}(\Omega')} \le \| \varphi \|_{L^{\infty}(\partial \Omega)}}$.
\smallskip

\item[(ii)]  $\displaystyle{
\| \nabla u \|_{L^{\infty}(\Omega' \setminus V)} \le C  \| \varphi \|_{C^{1,\alpha}(\partial \Omega)} },$
where $C$ depends only on $n$,  $\partial \Omega$ and $c$.
\end{enumerate}
\end{proposition}
\begin{proof}  For (i), since $u$ satisfies $\Delta u=0$ the maximum of $u$ must be attained at a boundary point $\partial \Omega'$.   As long as $u$ is not constant, the maximum cannot occur at $p \in \partial B^+ \cup \partial B^-$ by the strong maximum principle and the Hopf Lemma (see for example \cite[p. 347]{E}).   Hence $u$ achieves its maximum at a point of $\partial \Omega$ 
 so is bounded above by  $\| \varphi \|_{L^{\infty}(\partial \Omega)}$.  The lower bound of $u$ is similar.

For part (ii), by the usual interior estimates (see for example \cite[Theorem 3.9]{GT}) it is sufficient to obtain the bound for $|\nabla u|$ on neighborhoods of points in $\partial \Omega$ and neighborhoods of those points in $\partial B^+ \cup \partial B^-$ not contained in $V$.   To obtain these boundary estimates, which are local, one can compose with diffeomorphisms to flatten the boundary, at the expense of replacing the Laplace equation $\Delta u=0$ by uniformly elliptic equations of the form 
$$Lu = a^{ij} u_{ij} + b^i u_i=0,$$
with smooth coefficients $a^{ij}, b^i$ such that $(a^{ij}) \ge c \textrm{Id}$, for $c>0$.   Here and henceforth we are using subscripts (such as in  $u_i$, $u_{ij}$) to denote partial derivatives.

Using the method of freezing coefficients, and after a linear change of coordinates, we can further reduce to the case of the Poisson equation $\Delta u = g$, where $g$ is a lower order term.   In the case of neighborhoods of points of $\partial \Omega$, we have a  Dirichlet boundary condition $u = \varphi$ on a subset of the hyperplane $x_n=0$.  By standard Schauder estimates, which can be proved using Newtonian potentials and reflection, we have a local $C^{1,\alpha}$ estimate for $u$ in terms of the $L^{\infty}$ norm of $u$ and the $C^{1,\alpha}$ norm of $\varphi$.  See \cite[Corollary 8.36]{GT}, for example,  for more details.  This gives the required gradient estimate of $u$ in these neighborhoods.

A similar argument, with a suitable modification of the Newtonian potential, gives the boundary estimate in neighborhoods of points in $\partial B^+ \cup \partial B^-$ not contained in $V$, with the Dirichlet condition replaced by a Neumann boundary condition.    We refer the reader to  \cite[Theorem 6.30]{GT} (which although stated as a global estimate, is proved locally).  For an approach using Sobolev norms, see \cite[p.217]{Mi}. 
\end{proof}

Next we prove an elementary lemma on the normal derivative of $|\nabla u|^2$ on the boundary of $B^+$ (of course the case of $B^-$ is similar).
The unit normal vector to $B^+$, pointing towards the center of $B^+$ is given by
\begin{equation} \label{nu}
\nu = -(x_1, \ldots, x_{n-1}, x_n-1-\ve).
\end{equation}
We have:

\begin{lemma} \label{bdl}
At any point of $\partial B^+$ we have
$$\partial_{\nu} | \nabla u|^2 = 2 |\nabla u|^2,$$
where $\nu$ is the unit normal vector given by (\ref{nu}).
\end{lemma}
\begin{proof} The proof uses only the Neumann boundary condition $\partial_{\nu} u=0$.   Without loss of generality, translate the center of $B^+$ to the origin, and compute at the point $p=(0,\ldots, 0,-1)$.  Then $\nu(x) =-x$ and $\partial_{\nu}u=0$ becomes
\begin{equation} \label{nbc}
\sum_{i=1}^{n} x_i u_i=0,
\end{equation}
and in particular $u_n=0$ at $p$.  Write
$$x_n=x_n(x') = -\sqrt{1-x_1^2-\cdots - x_{n-1}^2}.$$
Differentiating (\ref{nbc}) with respect to $x_j$ for $j=1, \ldots, n-1$ and noting that $\partial_j x_n=0$ at $p$ we have
$$u_j - u_{nj} =0.$$
Hence, at $p$,
$$\partial_{\nu} |\nabla u|^2 = 2\sum_{i=1}^n u_i u_{in} = 2\sum_{i=1}^{n-1} u_i^2 = 2 |\nabla u|^2,$$
as required.
\end{proof}

We end this section by remarking that although Lemma \ref{bdl} makes use of the spherical shape of the inclusion, a similar result holds for more general inclusions which are strongly convex near the origin.  Indeed, suppose that we replace $B^+$ by a domain $D^+$ whose lower boundary near $x'=0$ is given by the graph $x_n = f(x')$ where $f$ is a smooth function with $f_1(0)=\cdots = f_{n-1}(0)=0$ and $f_{ij}(0)=P_{ij}$ for $(P_{ij})_{i,j=1}^{n-1}$ a positive definite symmetric matrix.  Then at a point on $\partial D^+$ near the origin we have
\begin{equation} \label{f}
\partial_{\nu} |\nabla u|^2 = 2 \sum_{i,j=1}^{n-1} \left( \frac{f_{ij}}{\sqrt{1+\sum_{k=1}^{n-1} f_k^2}} -f_if_j \right)u_iu_j + 2u_n^2.
\end{equation}
In particular, for $|x'|$ small,
\begin{equation} \label{Lemmanew} 
\partial_{\nu} |\nabla u|^2 = 2 \sum_{i,j=1}^{n-1} (P_{ij}+ E_{ij})u_iu_j + 2u_n^2, \quad \textrm{for } |E_{ij}| =O(|x'|^2).
\end{equation}
The eigenvalues of the matrix $P_{ij}$ are the principal curvatures of $\partial D^+$ at $(0, f(0))$.   The equation (\ref{Lemmanew}) can be regarded as a generalized version of Lemma \ref{bdl}.

 To prove (\ref{f}), observe that the inward pointing unit normal is given by
 $$\nu = \frac{1}{\sqrt{1+ \sum_{i=1}^{n-1} f_i^2 }} (-f_1, \ldots, -f_{n-1}, 1),$$
and so the
 Neumann condition $\partial_{\nu}u=0$ is equivalent to 
$$-\sum_{i=1}^{n-1} f_i u_i+u_n=0.$$
Differentiating this with respect to $x_j$ for $j=1, \ldots, n-1$ gives 
$$-\sum_{i=1}^{n-1} f_{ij} u_i -\sum_{i=1}^{n-1} f_i u_{ij} - \sum_{i=1}^{n-1}f_i u_{in} f_j + u_{nj} + u_{nn}f_j=0,$$
and then (\ref{f}) follows from a straightforward computation of $\partial_{\nu} |\nabla u|^2$, substituting from the two equalities above.

\section{Proof of Theorem \ref{mainthm}} \label{sectionproof}

In this section we prove Theorem \ref{mainthm}.

\begin{proof}[Proof of Theorem \ref{mainthm}]  We recall that $\gamma^*$ is by definition the positive solution of the quadratic equation (\ref{qe}).  Moreover, (\ref{qe}) has a negative solution.  Then
the assumption $0<\gamma <\gamma^*$ means that $\rho=\rho(\gamma)$ defined by
\begin{equation} \label{rho}
\rho:= -(n-2) \gamma^2 - (n^2-4n+5) \gamma + (n^2-5n+5)
\end{equation}
is strictly positive.  Choose $b>0$ sufficiently small such that
$$4\gamma\left(1+\frac{b}{4}\right)^{\gamma/2} < 4 \gamma + \frac{\rho}{n-1},$$
and then define $A>0$ by 
$$A= \frac{4\gamma + \frac{\rho}{n-1}}{b(2-\gamma)}.$$
Then $A, b$ are fixed constants depending only on $\gamma$ and $n$ satisfying the inequalities
\begin{equation} \label{Ab}
4\gamma \left(1+\frac{b}{4}\right)^{\gamma/2} < Ab (2-\gamma) < 4\gamma+\frac{2\rho}{n-1},
\end{equation}
which we will make use of later.

Let $\ve_0, c, \delta>0$ be small uniform constants, which later we may shrink if necessary.  For $0<\ve<\ve_0$, define 
\begin{equation} \label{Q}
Q = ((|x'|^2+\sigma)^{1-\gamma} + \ve^{1-\gamma(1-\delta)} - A(bx_n^2 + |x'|^4 +\sigma)^{1-\gamma/2}) |\nabla u|^2,
\end{equation}
on $$V = \ov{\Omega' \cap B_c},$$
for a small constant  $\sigma>0$ with $\sigma^{1-\gamma}<\ve^3$.
The constant $\sigma$ is only inserted so that $Q$ is a smooth function.  Later we will let $\sigma$ tend to zero.

We will use $C$ to denote a large positive  constant which may differ from line to line and is uniform in the sense that it will depend only on $n,  \partial \Omega, c,  \gamma, \delta$ and  $\ve_0$.  In particular $C$ will be independent of $\ve$, $x$ and $\sigma$.  

The upper boundary $\partial B^+ \cap B_c$ of $V$ is given by 
$x_n = f(x')$ where
$$f(x') = 1+ \ve - (1-|x'|^2)^{1/2},$$
which has $f(0)=\ve$, $f_j(0)=0$ and $f_{jk}(0) = \delta_{jk}$ giving 
\begin{equation} \label{xn}
x_n = \ve + \frac{|x'|^2}{2} + O(|x'|^3), \quad \textrm{on } \partial B^+ \cap B_c.
\end{equation}
Similarly, the lower boundary of $V$ is given by $x_n = -f(x') = - \ve-\frac{|x'|^2}{2} +O(|x'|^3)$ and hence in $V$ we have
\begin{equation}\label{xn2b}
|x_n| \le \ve + \frac{|x'|^2}{2} + O(|x'|^3).
\end{equation}

The quantity $Q$ achieves a maximum at a point $p$ in $V$.  If $p$ is in $\partial B_c$ then from Proposition \ref{pwk} we have $|\nabla u|^2(p) \le C\| \varphi \|^2_{C^{1,\alpha}(\partial \Omega)}$ and hence $Q(p) \le C \| \varphi \|^2_{C^{1,\alpha}(\partial \Omega)}$, giving
\begin{equation} \label{mainineq}
\left( (|x'|^2+\sigma)^{1-\gamma} + \ve^{1-\gamma(1-\delta)} \right)  |\nabla u|^2 \le C\| \varphi \|^2_{C^{1,\alpha}(\partial \Omega)}, \quad \textrm{on } V.
\end{equation}
Here we used (\ref{xn2b}) which implies that
\begin{equation} \label{util}
A(bx_n^2+ |x'|^4 + \sigma)^{1-\gamma/2} \le C(\ve^{2-\gamma} + |x'|^{4-2\gamma}) << 
(|x'|^2+\sigma)^{1-\gamma} + \ve^{1-\gamma(1-\delta)}.
\end{equation}

We will rule out the cases where $Q$ achieves a maximum at a point on the upper or lower boundaries of $V$ (on $\partial B^+$ or $\partial B^-$) or at an interior point of $V$.

First assume that $Q$ achieves a maximum at a point $p \in \partial B^+$.   Recalling that $\nu$ is given by (\ref{nu}) we have
\begin{equation} \label{dnu}
\begin{split}
\partial_{\nu} |x'|^2 ={} & -(x_1, \ldots, x_{n-1}, x_n-1-\ve) \cdot (2x_1, \ldots, 2x_{n-1}, 0) = -2|x'|^2\\
\partial_{\nu} x_n = {} &  -(x_1, \ldots, x_{n-1}, x_n-1-\ve) \cdot (0, \ldots, 0, 1) = 1+\ve-x_n.
\end{split}
\end{equation}
Using this and Lemma \ref{bdl}, we have, after shrinking $\ve_0$  and $c$ if necessary, writing $\eta=1/C$ for a uniform large $C$,
\begin{equation}\label{Qbdy2}
\begin{split}
0 \le {} &  \partial_{\nu} Q \\
= {} &\bigg(-(2-2\gamma)(|x'|^2+\sigma)^{-\gamma}|x'|^2 \\ {} & -A(2-\gamma) (bx^2_n + |x'|^4+\sigma)^{-\gamma/2} (bx_n
 \partial_{\nu} x_n - 2|x'|^4 ) \bigg) |\nabla u|^2 \\
{} &   +2\left( (|x'|^2+\sigma)^{1-\gamma} + \ve^{1-\gamma(1-\delta)} - A(bx_n^2 + |x'|^4+\sigma)^{1-\gamma/2}\right) | \nabla u|^2  \\
\le {} & \bigg(2\gamma (|x'|^2+\sigma)^{1-\gamma} + (2-2\gamma)(|x'|^2+\sigma)^{-\gamma}\sigma  \\ {} & -A(2-\gamma)(bx_n^2 + |x'|^4+\sigma)^{-\gamma/2} (bx_n(1+\ve-x_n)-2|x'|^4) \\ {}&  +2\ve^{1-\gamma(1-\delta)} \bigg) |\nabla u|^2  \\
 \le {} &  \bigg(2\gamma (|x'|^2+\sigma)^{1-\gamma} - Ab(2-\gamma) (bx_n^2+|x'|^4+\sigma)^{-\gamma/2} (1-\eta) x_n \\ {} &  +3\ve^{1-\gamma(1-\delta)}\bigg) |\nabla u|^2,
\end{split}
\end{equation}
where for the last line we used the inequalities
\[
\begin{split}
& (2-2\gamma)(|x'|^2+\sigma)^{-\gamma} \sigma \le 2\sigma^{1-\gamma} \le 2\ve^{3} \le \ve^{1-\gamma(1-\delta)} \\
& bx_n(1+\ve-x_n)-2|x'|^4 \ge b(1-\eta) x_n.
\end{split}
\]
Here and henceforth, we will use $\eta=1/C$ to denote a small positive constant which may differ from line to line, and which can be shrunk at the expense of shrinking $\ve_0$ or $c$.

Note that we may  neglect the term $3\ve^{1-\gamma(1-\delta)}$ on the right hand side of (\ref{Qbdy2}).  This is because from (\ref{xn}),
$$bx_n^2+ |x'|^4+\sigma\le Cx_n^2$$
and hence
$$(bx_n^2+|x'|^4+\sigma)^{-\gamma/2}x_n \ge \frac{x_n^{1-\gamma}}{C} \ge\frac{ \ve^{1-\gamma}}{C} >> \ve^{1-\gamma(1-\delta)}.$$
Then from (\ref{Qbdy2}) we obtain
\begin{equation}\label{Qbdy3}
\begin{split}
0 \le {} &  \partial_{\nu} Q 
 \le   \bigg(2\gamma |x'|^{2-2\gamma} - Ab(2-\gamma) (bx_n^2+|x'|^4)^{-\gamma/2} (1-\eta) x_n\bigg) |\nabla u|^2,
\end{split}
\end{equation}
after possibly changing $\eta$, and noting that the contribution of $\sigma$ can also absorbed by the good negative term.

We claim that 
\begin{equation} \label{claimeta}
(bx_n^2+ |x'|^4)^{-\gamma/2}x_n \ge   \frac{1}{2}\left(1+\frac{b}{4}\right)^{-\gamma/2}(1-\eta)  |x'|^{2-2\gamma}.
\end{equation}
To prove the claim we first note the elementary inequality
\begin{equation} \label{ei}
\left(b\left(\ve+\frac{y}{2}\right)^2+y^2\right)^{-\gamma/2} \left(\ve+\frac{y}{2}\right) \ge \frac{1}{2} \left(1+\frac{b}{4} \right)^{-\gamma/2} y^{1-\gamma}, \quad \textrm{for any } y, \ve \ge 0.
\end{equation}
Indeed to see (\ref{ei}) note that equality holds at $\ve=0$ and the left hand side of the equation is an increasing function of $\ve$.  But then (\ref{claimeta}) is an immediate consequence of (\ref{ei}) once we recall (\ref{xn}).

But from (\ref{Ab}) we have
$$\frac{1}{2} Ab(2-\gamma) \left(1+\frac{b}{4} \right)^{-\gamma/2} > 2\gamma.$$
Hence shrinking $\eta>0$ if necessary we can ensure that the right hand side of (\ref{Qbdy3}) is negative, contradicting the assumption that $Q$ achieves a maximum at a point $p \in \partial B^+$.  The case of $p\in \partial B^-$ is similar.

It remains to rule out the case when $Q$ achieves a maximum at a point $p$ in the interior of $V$.  Compute at $p$,
\begin{equation} \label{DQ1}
\begin{split}
0 \ge \Delta Q =  {} &\bigg\{ 2(1-\gamma)(n-1-2\gamma)(|x'|^2+\sigma)^{-\gamma}   + 4\gamma(1-\gamma) (|x'|^2+\sigma)^{-1-\gamma} \sigma \\
{} & + A(1-\gamma/2)(\gamma/2) (bx_n^2 + |x'|^4+\sigma)^{-1-\gamma/2} |\nabla (bx_n^2+|x'|^4)|^2 \\ {} & \mbox{}- A(1-\gamma/2)(bx_n^2 + |x'|^4+\sigma)^{-\gamma/2} \Delta (bx_n^2+ |x'|^4) \bigg\} |\nabla u|^2\\
{} & + 2((|x'|^2+\sigma)^{1-\gamma} + \ve^{1-\gamma(1-\delta)} - A(bx_n^2 + |x'|^4+\sigma)^{1-\gamma/2}) |\nabla \nabla u|^2 \\
{} & + 2 \nabla ((|x'|^2+\sigma)^{1-\gamma} - A(bx_n^2 + |x'|^4+\sigma)^{1-\gamma/2} ) \cdot \nabla |\nabla u|^2. 
\end{split}
\end{equation}
We wish to show that the right hand side is strictly positive, giving a contradiction.

 For the  negative term in the coefficient of $|\nabla u|^2$ we note that  $\Delta (bx_n^2+|x'|^4) \le 2b(1+\eta)$  and so
\begin{equation} \label{DQ2}
A(1-\gamma/2)(bx_n^2+|x'|^4+\sigma)^{-\gamma/2} \Delta (bx_n^2+|x'|^4) \le Ab(2-\gamma)(1+\eta)(|x'|^2+\sigma)^{-\gamma}.
\end{equation}
 In the above we used:
\begin{equation} \label{useful}
bx_n^2+|x'|^4+\sigma \ge |x'|^4+\sigma \ge (|x'|^2+\sigma)^2,
\end{equation}
since we may assume that $|x'|$ and $\sigma$ are  less than $1/2$.

We can essentially ignore the negative term in the coefficient of $|\nabla \nabla u|^2$ since, recalling (\ref{util}), we have
\begin{equation} \label{DQ3}
\begin{split}
A(bx_n^2+|x'|^4+\sigma)^{1-\gamma/2}  \le {} &  \eta \left( (|x'|^2+\sigma)^{1-\gamma} +   \ve^{1-\gamma(1-\delta)} \right).
\end{split}
\end{equation}
Combining (\ref{DQ1}), (\ref{DQ2}), (\ref{DQ3}) and discarding some nonnegative terms we have
\begin{equation} \label{DQ4}
\begin{split}
\Delta Q \ge  {} & \bigg( 2(1-\gamma) (n-1-2\gamma) - Ab(2-\gamma)(1+\eta) \bigg) (|x'|^2 +\sigma)^{-\gamma} |\nabla u|^2 \\
{} & + 2(1-\eta)\bigg( (|x'|^2+\sigma)^{1-\gamma} + \ve^{1-\gamma(1-\delta)} \bigg) |\nabla \nabla u|^2 \\
{} & + 2 \nabla ((|x'|^2+\sigma)^{1-\gamma} - A(bx_n^2 + |x'|^4+\sigma)^{1-\gamma/2} ) \cdot \nabla |\nabla u|^2. 
\end{split}
\end{equation}

It remains to control the term
\begin{equation} \label{star}
(*)=2 \nabla ( (|x'|^2+\sigma)^{1-\gamma} - A(bx_n^2 + |x'|^4+\sigma)^{1-\gamma/2} ) \cdot \nabla |\nabla u|^2.
\end{equation}
First observe that
$$\nabla |x'|^2 = \left( 2x_1, \ldots, 2x_{n-1}, 0 \right),$$
and we may make a change of coordinates so that $x_2 = \cdots = x_{n-1}=0$ and $x_1\ge 0$ and hence
$$\nabla |x'|^2 = \left( 2|x'|, 0, \ldots, 0 \right).$$
Next, we use the fact that at the maximum of $Q$ we have
\[
\begin{split}
0= Q_1 = {} &  |\nabla u|^2 \nabla_1  ((|x'|^2+\sigma)^{1-\gamma} - A(bx_n^2 + |x'|^4+\sigma)^{1-\gamma/2} )  \\
{} & +  ( (|x'|^2+\sigma)^{1-\gamma} +\ve^{1-\gamma(1-\delta)} - A(bx_n^2 + |x'|^4+\sigma)^{1-\gamma/2} ) \nabla_1 |\nabla u|^2. \\
\end{split}
\]
Hence
\begin{equation} \label{s1}
\begin{split}
\lefteqn{2 \nabla_1 ((|x'|^2+\sigma)^{1-\gamma} - A(bx_n^2+|x'|^4+\sigma)^{1-\gamma/2}) \nabla_1 |\nabla u|^2 } \\ 
= {} & - 2( (|x'|^2+\sigma)^{1-\gamma} +\ve^{1-\gamma(1-\delta)} - A(bx_n^2 + |x'|^4+\sigma)^{1-\gamma/2} ) \frac{ \left( \nabla_1 |\nabla u|^2 \right)^2}{|\nabla u|^2}  \\
= {} & -8 ((|x'|^2+\sigma)^{1-\gamma} +\ve^{1-\gamma(1-\delta)} - A(bx_n^2 + |x'|^4+\sigma)^{1-\gamma/2} ) \frac{ \left( \sum_j u_j u_{j1} \right)^2 }{|\nabla u|^2} \\
\ge {} & - 8(1+\eta)  \bigg( (|x'|^2+\sigma)^{1-\gamma} +\ve^{1-\gamma(1-\delta)} \bigg) \sum_j u_{j1}^2,
\end{split}
\end{equation}
where for the last line we used (\ref{DQ3}) and the Cauchy-Schwarz inequality.  Note that we may assume without loss of generality that $|\nabla u|^2\neq 0$.
We can replace $\sum_j u_{j1}^2$ above by $\frac{n-1}{n} |\nabla \nabla u|^2$ using the fact that $u$ is harmonic:
\begin{equation} \label{s4}
\begin{split}
\sum_{j=1}^n u_{j1}^2 = {} & \frac{n-1}{n} u_{11}^2 + \frac{1}{n} u_{11}^2 + \sum_{j=2}^n u_{j1}^2 \\
\le {} & \frac{n-1}{n} u_{11}^2 + \frac{1}{n} (u_{22}+\cdots + u_{nn})^2 + \frac{1}{2} \sum_{i\neq j}^n u_{ij}^2 \\
\le {} & \frac{n-1}{n} u_{11}^2 + \frac{n-1}{n} \sum_{j=2}^n u^2_{jj} + \frac{n-1}{n}\sum_{i \neq j}^n u_{ij}^2, \\
= {} & \frac{n-1}{n} \sum_{i,j=1}^n u_{ij}^2 = \frac{n-1}{n} |\nabla \nabla u|^2,
\end{split}
\end{equation}
where for the third line we used
 the elementary inequality that for real numbers $a_1, \ldots, a_{n-1}$ we have 
$$\left( \sum_{\ell=1}^{n-1} a_{\ell}\right)^2 \le (n-1) \sum_{\ell=1}^{n-1} a^2_{\ell}.$$

On the other hand we can estimate the same quantity as in (\ref{s1}) in a different way:
\begin{equation} \label{s2}
\begin{split}
\lefteqn{2 \nabla_1 ( (|x'|^2+\sigma)^{1-\gamma} - A(bx_n^2+|x'|^4+\sigma)^{1-\gamma/2}) \nabla_1 |\nabla u|^2 } \\ = {} & -2 |\nabla u|^2 \frac{ \bigg( \nabla _1 ((|x'|^2+\sigma)^{1-\gamma} - A(bx_n^2 + |x'|^4+\sigma)^{1-\gamma/2}) \bigg)^2}{(|x'|^2+\sigma)^{1-\gamma} +\ve^{1-\gamma(1-\delta)} - A(bx_n^2+|x'|^4+\sigma)^{1-\gamma/2}} \\
\ge {} & - 2(1+\eta)(2-2\gamma)^2 \frac{ (|x'|^2+\sigma)^{-2\gamma}|x'|^2}{(|x'|^2+\sigma)^{1-\gamma}} |\nabla u|^2 \\
\ge {}& - 8(1+\eta) (1-\gamma)^2 (|x'|^2+\sigma)^{-\gamma} | \nabla u|^2,
\end{split}
\end{equation}
for any small $\eta>0$ where for the third line we used (\ref{DQ3}) and the bounds
$$\nabla_1 (|x'|^2+\sigma)^{1-\gamma} = (2-2\gamma) (|x'|^2+\sigma)^{-\gamma}|x'|$$ and
  $$A|\nabla_1 (bx_n^2+|x'|^4+\sigma)^{1-\gamma/2}| \le \eta \nabla_1 (|x'|^2+\sigma)^{1-\gamma}.$$

Next, we note that 
\begin{equation} \label{s3}
\begin{split}
\lefteqn{ 2 \nabla_n ((|x'|^2+\sigma)^{1-\gamma} -A(bx_n^2+|x'|^4+\sigma)^{1-\gamma/2}) \nabla_n |\nabla u|^2 } \\
= {} & -4Ab(2-\gamma)(bx_n^2+|x'|^4+\sigma)^{-\gamma/2} x_n \sum_ju_j u_{jn} \\
\ge {} & - \eta (bx_n^2+|x'|^4+\sigma)^{-\gamma/2} |\nabla u|^2 - \frac{C}{\eta} |x_n|^2 (bx_n^2+|x'|^4+\sigma)^{-\gamma/2} |\nabla \nabla u|^2 \\
\ge {} & - \eta (|x'|^2+\sigma)^{-\gamma} |\nabla u|^2 - \eta \left( (|x'|^2+\sigma)^{1-\gamma} + \ve^{1-\gamma(1-\delta)} \right)  |\nabla \nabla u|^2,
\end{split}
\end{equation}
where for the last line we used (\ref{useful}) and
\[
\begin{split}
|x_n|^2 (bx_n^2+|x'|^4+\sigma)^{-\gamma/2} \le  {} & b^{-\gamma/2}|x_n|^{2-\gamma}  \\ \le {} &  2b^{-\gamma/2}\left(\ve + \frac{|x'|^2}{2} \right)^{2-\gamma} \\  \le {} &  \frac{\eta^2}{C} \bigg((|x'|^2+\sigma)^{1-\gamma} + \ve^{1-\gamma(1-\delta)}\bigg).
\end{split}
\]

Then combining (\ref{star}), (\ref{s1}), (\ref{s4}), (\ref{s2}), (\ref{s3}) we obtain for  $0<\alpha<1$ to be determined:
\[
\begin{split}
(*) \ge {} & - 8(1-\alpha +\eta)(1-\gamma)^2 (|x'|^2+\sigma)^{-\gamma} |\nabla u|^2  \\
{} & - (8\alpha + \eta)((|x'|^2+\sigma)^{1-\gamma}+\ve^{1-\gamma(1-\delta)}  ) \frac{n-1}{n} |\nabla \nabla u|^2. 
\end{split}
\]
Then pick $\alpha$ so that 
$8\alpha\frac{n-1}{n}  < 2,$ in order for the second term on the right hand side to be absorbed by the $|\nabla \nabla u|^2$ term in (\ref{DQ4}).  Namely we choose
$$\alpha = \frac{n}{4(n-1)} -\eta.$$   Hence from (\ref{DQ4}) we obtain
 at the maximum of $Q$,  recalling (\ref{rho}),
\[
\begin{split}
 0 \ge {} &  \Delta Q  \\  \ge {} & \bigg(2(1-\gamma)(n-1-2\gamma) - Ab(2-\gamma)  \\ {} &  -8(1-\gamma)^2 \left( 1- \frac{n}{4(n-1)}\right) -\eta \bigg)  (|x'|^2+\sigma)^{-\gamma} |\nabla u|^2 \\
 = {} & \frac{2}{n-1} \bigg( \rho  - \frac{(n-1)(Ab(2-\gamma)-4\gamma)}{2} - \frac{(n-1)\eta}{2} \bigg) (|x'|^2+\sigma)^{-\gamma} |\nabla u|^2 >0,
\end{split}
\]
 a contradiction.  For the last inequality we used (\ref{Ab}) and chose $\eta>0$ sufficiently small. 

Hence we have ruled out the possibility that $Q$ obtains a maximum at an interior point of $V$ or at a boundary point intersecting with $\partial B^+$ or $\partial B^-$.  Hence $Q$ achieves a maximum at a point on $\partial B_c$ and (\ref{mainineq}) holds.  Let $\sigma \rightarrow 0$ and observe that $\delta>0$ can be made arbitrarily small.  Theorem \ref{mainthm} follows. \end{proof}

\section{A proof of the estimate of Bao-Li-Yin} \label{sectionBLY}

In this section we briefly explain how a simplified version of our method gives a different proof of the estimate (\ref{BLY}) of Bao-Li-Yin \cite{BLY} for the case of spherical inclusions of the same radii.

\begin{theorem} \label{thmBLY} For $u$ solving (\ref{mainequation}) as in the introduction,
\begin{equation} \label{BLY2}
| \nabla u |(x) \le \frac{C \| \varphi \|_{C^{1,\alpha}(\partial \Omega)}}{ (\ve+ |x'|^2)^{1/2} }, \quad \textrm{for } x\in \Omega',
\end{equation}
for $C$ depending only on $n$ and $\Omega$.
\end{theorem}
\begin{proof}  Since the result is a simpler version of Theorem \ref{mainthm}, we provide here just a sketch of the proof.  We use $V$, $B_c$ as above.  Define
$$Q = (|x'|^2+\ve - 2x_n^2) |\nabla u|^2 + Au^2,$$
for a uniform constant $A$  to be determined.  Recalling (\ref{xn2b}),  an upper bound $Q\le C \| \varphi \|^2_{C^{1,\alpha}(\partial \Omega)}$ implies (\ref{BLY2}), after shrinking $c$ if necessary.

First assume $Q$ achieves its maximum on $V$ at a point $p$ in $\partial B^+ \cap \textrm{int}(B_c)$.  Then using Lemma \ref{bdl}, (\ref{dnu}) and (\ref{xn}) we have, after possibly shrinking $c$,
\begin{equation} \label{Q1}
\begin{split}
0 \le \partial_{\nu} Q = {} & (-4x_n(1+\ve -x_n) +2\ve - 4x_n^2) |\nabla u|^2<0,
\end{split}
\end{equation}
a contradiction.  The case when $p$ is in $\partial B^- \cap \textrm{int}(B_c)$ is similar.

Next suppose that $p$ is an interior point of $V$.  Then computing at $p$ we have
\begin{equation} \label{DeltaQ1}
\begin{split}
0 \ge  \Delta Q  =  {} & 2(n-3) |\nabla u|^2 + 2(|x'|^2+\ve-2x_n^2) |\nabla \nabla u|^2 \\ {} & + 2\nabla (|x'|^2-2x_n^2) \cdot \nabla |\nabla u|^2 + 2A |\nabla u|^2.
\end{split}
\end{equation}

But 
$$2\nabla |x'|^2 \cdot \nabla |\nabla u|^2 \ge -8 |x'| | \nabla u | | \nabla \nabla u| \ge-  |x'|^2 |\nabla \nabla u|^2 - 16 |\nabla u|^2,$$
and 
$$-4 \nabla x_n^2 \cdot \nabla |\nabla u|^2 \ge - 8 x_n^2 |\nabla \nabla u|^2 - 8 |\nabla u|^2.$$
In (\ref{DeltaQ1}) we obtain at $p$,
$$0 \ge \Delta Q \ge 2(A+n-15)|\nabla u|^2 + (|x'|^2+2\ve - 12x_n^2)|\nabla \nabla u|^2 >0,$$
as long as we choose $A>15-n$, since we may assume  that $12x_n^2 \le |x'|^2+2\ve$.  This is a contradiction. 

It follows that $Q$ achieves its maximum on $V$ at a boundary point on $\partial B_c$, and the result follows from Proposition \ref{pwk}.
\end{proof}

We end  by observing that an immediate consequence of (\ref{BLY2})  is a uniform bound on the derivative of $u$ in the $x_n$ direction.

\begin{corollary} \label{cor}
Let $u$ solve (\ref{mainequation}) as in the introduction. For $x \in B_{1/2} \cap \Omega'$ where $B_{1/2}$ is the ball of radius $1/2$ centered at $0$, we have
\begin{equation} \label{un}
|u_n(x)|\le C \| \varphi \|_{C^{1,\alpha}(\partial \Omega)}, \end{equation}
for a constant $C$ depending only on $n$ and $\Omega$.
\end{corollary}
\begin{proof}
Since $u_n$ is harmonic, $|u_n|$ does not achieve a maximum in the interior of $\Omega'$ unless it is constant.  Hence, applying Proposition \ref{pwk}, we only need to bound $u_n$ on $\partial B^+$ near the origin (the case $\partial B^-$ is similar).  Recalling (\ref{nu}), the Neumann boundary condition gives 
$$0=u_{\nu} = -\sum_{i=1}^{n-1} x_i u_i - (x_n-1-\ve)u_n.$$
But from (\ref{xn2b}) we have $|x_n-1-\ve| \ge 1/2$ so using (\ref{BLY2}),
$$|u_n| \le 2 |x'| |\nabla u| \le C \| \varphi \|_{C^{1,\alpha}(\partial \Omega)},$$
completing the proof.
\end{proof}

\bigskip
\noindent
{\bf Acknowledgements.} \ The author thanks the referees for  helpful comments and suggestions.


\begin{thebibliography}{0}
\bibitem{ACKLY} Ammari, H., Ciraolo, G., Kang, H., Lee, H.,  Yun, K., {\em Spectral analysis of the Neumann-Poincar\'e operator and characterization of the stress concentration in anti-plane elasticity}, Arch. Ration. Mech. Anal. 208 (2013), no. 1, 275--304
\bibitem{AKLLL} Ammari, H., Kang, H., Lee, H., Lee, J., Lim, M., {\em 
Optimal estimates for the electric field in two dimensions.}, 
J. Math. Pures Appl. (9) 88 (2007), no. 4, 307--324
\bibitem{AKL}  Ammari, H., Kang, H., Lim, M., {\em Gradient estimates for solutions to the conductivity problem}, Math. Ann. 332 (2005), no. 2, 277--286
\bibitem{BASL} Babu\v{s}ka, I., Andersson, B., Smith, P.J., Levin, K., {\em Damage analysis of fiber composites. I. Statistical analysis on fiber scale}, Comput. Methods Appl. Mech. Engrg. 172 (1999), no. 1-4, 27--77
\bibitem{BLY0} Bao, E., Li, Y.Y., Yin, B., {\em Gradient estimates for the perfect conductivity problem}, Arch. Ration. Mech. Anal. 193 (2009), no. 1, 195--226
\bibitem{BLY} Bao, E.,  Li, Y.Y., Yin, B., {\em Gradient estimates for the perfect and insulated conductivity problems with multiple inclusions}, Comm. Partial Diff. Equations 35 (2010), no. 11, 1982--2006
\bibitem{BV} Bonnetier, E., Vogelius, M., {\em An elliptic regularity result for a composite medium with "touching'' fibers of circular cross-section}, SIAM J. Math. Anal. 31 (2000), no. 3, 651--677
\bibitem{BC} Budiansky, B., Carrier, G.F., {\em High Shear Stresses in Stiff-Fiber Composites},  
J. Appl. Mech.  51 (1984), no 4, 733--735
\bibitem{DL} Dong, H., Li, H., {\em 
Optimal estimates for the conductivity problem by Green's function method}, Arch. Ration. Mech. Anal. 231 (2019), no. 3, 1427--1453
\bibitem{DLY} Dong, H., Li, Y.Y., Yang, Z., {\em Optimal gradient estimates of solutions to the insulated conductivity problem in dimension greater than two}, preprint, arXiv:2110.11313v1
\bibitem{E} Evans, L.C., {\em Partial differential equations}, Second edition. Graduate Studies in Mathematics, 19. American Mathematical Society, Providence, RI, 2010
\bibitem{GT} Gilbarg, D., Trudinger, N.S., {\em Elliptic partial differential equations of second order. Reprint of the 1998 edition.}, Classics in Mathematics. Springer-Verlag, Berlin, 2001
\bibitem{KLY} Kang, H., Lim, M., Yun, K., {\em Asymptotics and computation of the solution to the conductivity equation in the presence of adjacent inclusions with extreme conductivities},  J. Math. Pures Appl. (9) 99 (2013), no. 2, 234--249
\bibitem{K} Keller, J.B., {\em Stresses in Narrow Regions},  J. Appl. Mech. 60 (1993), no 4, 1054--1056 
\bibitem{LN} Li, Y.Y. , Nirenberg, L., {\em Estimates for elliptic systems from composite material}, 
Dedicated to the memory of J\"urgen K. Moser.
Comm. Pure Appl. Math. 56 (2003), no. 7, 892--925
\bibitem{LV} Li, Y.Y., Vogelius, M., {\em Gradient estimates for solutions to divergence form elliptic equations with discontinuous coefficients.},
Arch. Ration. Mech. Anal. 153 (2000), no. 2, 91--151
\bibitem{LY} Li, Y.Y. and Yang, Z., {\em Gradient estimates of solutions to the insulated conductivity problem in dimension greater than two}, preprint, arXiv: 2012.14056 
\bibitem{LYu} Lim, M., Yu, S., {\em 
Stress concentration for two nearly touching circular holes}, preprint, arXiv:1705.10400
\bibitem{M} Markenscoff, X., {\em Stress amplification in vanishingly small geometries}, Computational Mechanics 19 (1996), no. 1, 77--83
\bibitem{Mi} Mikha\v{i}lov, V.P., {\em Partial differential equations}. Translated from the Russian by P. C. Sinha. ``Mir'', Moscow; distributed by Imported Publications, Inc., Chicago, Ill., 1978
\bibitem{Y1} Yun, K., {\em  Estimates for electric fields blown up between closely adjacent conductors with arbitrary shape}, SIAM J. Appl. Math. 67 (2007), no. 3, 714--730
\bibitem{Y2}  Yun, K., {\em Optimal bound on high stresses occurring between stiff fibers with arbitrary shaped cross-sections}, J. Math. Anal. Appl. 350 (2009), no. 1, 306--312
\bibitem{Y3} Yun, K., {\em An optimal estimate for electric fields on the shortest line segment between two spherical insulators in three dimensions}, 
J. Differential Equations 261 (2016), no. 1, 148--188
\end{thebibliography}
\end{document}